\documentclass[11pt]{amsart}
\usepackage{amssymb,latexsym,amsfonts}
\usepackage{amsmath}
\usepackage{amscd}
\usepackage{pb-diagram,pb-xy}%lamsarrow,pb-lams}

\usepackage{graphicx}
\usepackage{hyperref}
\usepackage{color}
\usepackage[matrix,arrow]{xy}

\usepackage{fullpage}
\newcommand{\Ein}{{\mathrm {Ein}}}

\newcommand{\Eink}{{\mathrm {Ein}_k}}

\newcommand{\Riem}{{\mathrm {Riem}}}
\newcommand{\bRiem}{{\mathbf {Riem}}}
\newcommand{\riem}{{\mathrm {riem}}}
\newcommand{\briem}{{{\mathbf {riem}}}}
\newcommand{\Riemt}{{\mathrm {Riem}_t}}

\newcommand{\Riems}{{\mathrm {Riem}_s}}

\newcommand{\Scal}{{\rm Scal}}

\newcommand{\Ric}{{\rm Ric}}
\newcommand{\Sec}{{\rm Sec}}

\DeclareMathAlphabet{\mathpzc}{OT1}{pzc}{m}{it}

\def\lambdam{\lambda_{{\rm max}}}

\newtheorem{theorem}{Theorem}[section]
\newtheorem{corollary}[theorem]{Corollary}

\newtheorem{proposition}[theorem]{Proposition}
\newtheorem*{thm-a}{Theorem\! A}
\newtheorem*{thm-aa}{Theorem\! ${\rm A}^{\prime}$}
\newtheorem*{cor-bb}{Corollary\! ${\rm B}^{\prime}$}
\newtheorem*{thm-cc}{Theorem\! ${\rm C}^{\prime}$}
\newtheorem*{thm-bb}{Theorem\! ${\rm B}^{\prime}$}
\newtheorem*{cor-a}{Corollary\! {\rm A}}
\newtheorem*{thm-b}{Theorem\! {\rm B}}
\newtheorem*{cor-b}{Corollary\! {\rm B}}
\newtheorem*{thm-c}{Theorem\! {\rm C}}
\newtheorem*{cor-c}{Corollary\! {\rm C}}
\newtheorem*{thm-d}{Theorem\! {\rm }D}
\newtheorem*{thm-dd}{Theorem\! ${\rm D}^{\prime}$}
\newtheorem*{cor-dd}{Corollary\! ${\rm D}^{\prime}$}
\newtheorem*{cor-d}{Corollary\! {\rm D}}
\newtheorem*{thm-e}{Theorem\! {\rm E}}
\newtheorem*{cor-e}{Corollary\! {\rm E}}
\newtheorem*{thm-f}{Theorem\! {\rm F}}
\newtheorem*{cor-f}{Corollary\! F}
\newtheorem*{thm-g}{Theorem\! {\rm G}}
\newtheorem*{conjecture}{Conjecture}

\newtheorem*{conj-d}{Conjecture D}
\newtheorem*{conj-c}{Conjecture C}
\newtheorem{thm}{Theorem}[section]
\theoremstyle{definition}
\newtheorem{definition}[thm]{Definition}

\newtheorem*{remark}{Remark}

\newtheorem{example}{Example}[section]

\begin{document}
\author{Mohammed Larbi Labbi} 
\address{Department of Mathematics\\
 College of Science\\
  University of Bahrain\\
  32038, Bahrain.}
\email{mlabbi@uob.edu.bh}
%\subjclass[2020]{53C21, 53C20}
\renewcommand{\subjclassname}{
  \textup{2020} Mathematics Subject Classification}

\title{On a stratification of positive scalar curvature compact manifolds}  \maketitle
 \begin{abstract}
For a compact PSC Riemannian $n$-manifold $(M,g)$, the metric constant $\mathrm {Riem}(g)\in (0, \binom{n}{2}]$ is defined to be the infinimum over $M$ of the spectral scalar curvature $\frac{\sum_{i=1}^N\lambda_i}{\lambda_{\rm max}}$ of $g$, where $\lambda_1, ...,\lambda_N$ are the eigenvalues of the curvature operator of $g$ and $\lambda_{\rm max}$ is the maximal eigenvalue. The functional $g\to \mathrm {Riem}(g)$ is continuous, re-scale invariant and defines a stratification of the space of PSC metrics over $M$.
 We introduce as well the smooth constant $\mathbf {Riem}(M)\in (0, \binom{n}{2}]$, which is the supremum of $\mathrm {Riem}(g)$ over the set of all psc Riemannian metrics $g$ on $M$. \\
In this paper, we show that in the top layer, compact manifolds with $\mathbf{Riem}=\binom{n}{2}$ are positive space forms. No manifolds have their $\mathbf{Riem}$ in the interval $(\binom{n}{2}-2, \binom{n}{2})$.  The manifold $S^{n-1}\times S^1$ and  arbitrary connected sums of copies of it with connected sums of positive  space forms all have  $\mathbf{Riem}=\binom{n-1}{2}$.
For $1\leq p\leq n-2\leq 5$, we prove that the manifolds $S^{n-p}\times T^p$ take the intermediate values   $\mathbf {Riem}=\binom{n-p}{2}$.
From the bottom, we prove that simply connected (resp. $2$-connected, $3$-connected and non-string) compact manifolds of dimension $\geq 5$ (resp. $\geq 7$, $\geq 9$)  have $\mathbf{Riem}\geq 1$ (resp. $\geq 3$, $\geq 6$).  The proof of these last three results is based on surgery. In fact,  we prove that the smooth $\mathbf{Riem}$ constant doesn't decrease after a surgery on the manifold with adequate codimension.\\

\end{abstract}

\tableofcontents
\section{Introduction and statement of the main results}

\subsection{Introduction}
Throughout this paper, $(M,g)$ denotes a closed connected Riemannian manifold of dimension $n$. We denote   by $R$ and $\Scal$ respectively the Riemann curvature tensor and  the scalar curvature of $(M,g)$. We define a string of curvature tensors on  $(M,g)$ as follows
\begin{equation}
\Riemt (g):=\Scal \, \frac{g^2}{2} -2tR,
\end{equation}
Where $t$ is a constant real number, and $g^2$ is the square of $g$ with respect to the Kulkarni-Nomizu product. These curvatures are defined in a similar way as the tensors $\Eink$ in \cite{modified-Eink}.
We are interested in  the positivity properties of these curvature tensors. In order to ensure that their positivity is stronger then positive scalar curvature, we restrict the parameter $t$ to be less than $n(n-1)/2$. In fact, one can easily show that the full trace of $\Riemt$ equals $\bigl((n(n-1)-2t\bigr)\Scal.$\\
We remark that for a compact Riemannian $n$-manifold $(M,g)$ with positive scalar curvature and for $0<t<n(n-1)/2$, the tensor $\Riemt (g)$ is positive definite if and only if at each point of $M$ one has
\begin{equation}\label{lambdam}
t<\frac{\Scal}{2\lambdam},
\end{equation}
 where $\lambdam$ denotes the maximal eigenvalue of the Riemann curvature operator at the corresponding point. \\

A straightforward consequence of the above characterization of the positivity of the  tensors $\Riemt (g)$ is the following descent positivity property
\begin{equation}
{\rm For}\,\, 0<t<s<n(n-1)/2,\,\,\, \Riems>0 \Rightarrow \Riemt>0\Rightarrow \Scal >0.
\end{equation}

We therefore define the metric invariant 
\begin{equation}
\Riem(g):=\sup\{t \in (0, n(n-1)/2: \Riemt(g)>0\}.
\end{equation}
We set $\Riem(g)=0$  if the scalar curvature of $g$ is not positive. An immediate consequence of property (\ref{lambdam}), is that for a metric $g$ of positive scalar curvature one has
\begin{equation}
\Riem(g)=\inf_{M}\frac{\Scal(g)}{2\lambda_{{\rm max}}(g)}.
\end{equation}
The above formula shows in particular that the metric invariant $\Riem(g)$ is re-scale invariant, That is for any positive real number $t$, one has
\begin{equation}
\Riem(tg)=\Riem(g).
\end{equation}
\begin{example}
\begin{enumerate}
\item For a compact Riemannian $n$-manifold $(M,g)$ of positive constant sectional curvature, one has $\Riem(g)=\frac{n(n-1)}{2}.$ It has the maximal possible value. Conversely, if $\Riem(g)=\frac{n(n-1)}{2}$ then $g$ has positive constant sectional curvature.
\item For the standard product metric $g$ on the product $S^{n-p}\times T^p$  of a round sphere of curvature 1 with a flat torus one has $\Riem(g)=\frac{(n-p)(n-p-1)}{2}$.
\item For the standard product metric $g$ on the product $S^{n-p}\times H^p$  of a round sphere  of curvature $1$ with a hyperbolic space of curvature $-1$  one has $\Riem(g)=\frac{(n-1)(n-2p)}{2}$.
\item For the standard Fubiny-Study metric $g$ on $\mathbb{CP}^n$ one has  $\Scal(g)=4n(n+1)$ and $\lambda_{{\rm max}}(g)=2n+2$ \cite{Bour-Kar}, and consequently $\Riem(g)=n$.
\item For the standard Fubiny-Study metric $g$ on $\mathbb{HP}^n$ one has $\Scal(g)=16n(n+2)$ and $\lambda_{{\rm max}}(g)=4n$ \cite{Bour-Kar}, and consequently $\Riem(g)=2n+4$.
\end{enumerate}
\end{example}

The metric ${\rm Riem}$ invariant defines a {\emph{pre-order}} on the set of Riemannian metrics on $M$ as follows
\begin{equation}
g_1\preceq g_2 \,\,\, {\rm if}\,\,\, \Riem(g_1)\leq \Riem(g_2).
\end{equation}
This lead us naturally to the study of maximal metrics with respect to the above pre-order. \\
We define then a smooth invariant $\Riem(M)$ to be
\begin{equation}
\bRiem(M)= \sup\{\Riem(g)\colon  g\in {\mathcal M}\},
\end{equation}
where ${\mathcal M}$ denotes the space of all  Riemannian metrics on $M$. We set $\bRiem(M)=0$ if $M$ has no psc metrics. For the seek to provide the reader a better feeling of this invariant, we provide (prematurely) the following  examples whose proofs will be provided later in this paper.
\begin{example}
\begin{enumerate}
\item For all $n\geq 2$, It is easy to see that $\bRiem(S^n)=\bRiem(\mathbb{RP}^n)=\binom{n}{2}.$
\item For $r\geq 2$ and $n\geq 3$, connected sums of $r$-copies of $(\mathbb{RP}^n)$ have $\bRiem=\binom{n-1}{2}$.
\item For all $n\geq 3$, $\bRiem(S^{n-1}\times S^1)=\binom{n-1}{2}$. Furthermore, the connected sum of arbitrary $r$ copies of the previous manifold satisfies $\bRiem(\#_r(S^{n-1}\times S^1))=\binom{n-1}{2}$.
\item For $0\leq p\leq n-2\leq 5$, $\bRiem(S^{n-p}\times T^p)=\binom{n-p}{2}$.
\end{enumerate}
\end{example}
At this stage, we introduce the following terminology
\begin{definition}
 We shall say that a Riemannian metric $g$ on $M$ is \emph{maximal} or \emph{($\Riem$-maximal)} if $\Riem(g)\geq \Riem(h)$ for any Riemannian metric $h$ on $M$.

\end{definition}
For instance, in the above example the standard  product metric on $S^{n-1}\times S^1$ is maximal, while  the connected sum $\#_r(S^{n-1}\times S^1)$ has no maximal metrics for $r\geq n$.\\

In the literature one finds several notions of intermediate curvature conditions which are useful in different directions. The  author believes that the curvature conditions that are {\emph stable under surgeries, satisfy the ideal property with respect to the Cartesian product of manifolds and have a fiber bundle property (as below) are subtle and deserve a special attention. The point is that the above three properties are the keys to convert  geometrical classification problems to  topological ones. The typical example is of course positive scalar curvature condition and this approach was initiated by Gromov and Lawson  \cite{GroLa}.\\
 We will prove in the first part of this  paper  that the condition  of the existence of a Riemannian metric with $\Riem(g)>t$ on a given compact manifold satisfies all the three conditions. In the second part we shall show that  the same condition $\Riem>t$ imposes several restrictions on the topology of the manifold. 

\subsection{Statement of the main results of the paper}
The first result shows that the condition $\bRiem>t$ satisfies an ideal property with respect to Cartesian products
\begin{proposition}{\label{idealp}}
Let $M_1$ and $M_2$ be compact manifolds, then
$$\bRiem(M_1\times M_2) \geq \max\{ \bRiem(M_1), \bRiem(M_2)\}.$$
\end{proposition}

The above property is still true for general Riemannian submersions. In particular, we show that the condition $\bRiem>t$ satisfies a fibre bundle property
\begin{thm-a}
Let $\pi:M\to B$ be a fibre bundle with fibre $F$ of dimension $p$ and structure group $G$. Suppose $M$ and $B$ are compact and that the fibre $F$ admits a $G$-invariant metric $\hat{g}$ with $\Riem(\hat{g})>t$ for some $t\in [0, p(p-1)/2)$. Then $\bRiem(M)>t$.
\end{thm-a}
In particular, if compact manifolds $M$ and $N$ are respectively the total spaces of a $\mathbb{CP}^2$  bundle and an $\mathbb{HP}^2$ bundle with structure group the group of isometries of their Fubiny-Study metrics, then   $\bRiem(M) \geq 2$ and $\bRiem(N) \geq 8$.\\
The next result is the  surgery theorem
\begin{thm-b}\label{thmb}
\begin{itemize}
\item[a)] Let $q$ be an integer such that $3\leq q\leq n$ and $M$ be a compact $n$-manifold with $ \bRiem(M)\leq \binom{q-1}{2}$.
 If a compact manifold $\widehat{M}$ is obtained from $M$ by a surgery of 
codimension $\geq q$ then $$\bRiem(\widehat{M})\geq \bRiem(M).$$
\item[b)] Let $k$ be an integer with $2\leq k\leq n-1$ and $M$ be a compact $n$-manifold with $ \bRiem(M)> \binom{k}{2}$.
If a compact manifold $\widehat{M}$ is obtained from $M$ by a surgery of 
codimension $\geq k+1$ then $$\bRiem(\widehat{M})\geq \binom{k}{2}.$$
\end{itemize}
\end{thm-b}
One important consequence of the previous results is the following  gap theorem

\begin{thm-bb}
\begin{enumerate}
\item For a compact simply connected manifold $M$ of dimension $\geq 5$ one has: 
\[\bRiem(M)>0\implies \bRiem(M) \geq 1.\]
\item For a compact $2$-connected manifold $M$ of dimension $\geq 7$ one has: \[\bRiem(M)>0\implies \bRiem(M) \geq 3.\]
\item For a compact $3$-connected and non-string manifold $M$ of dimension $\geq 9$ one has: \[\bRiem(M)>0\implies \bRiem(M) \geq 6.\]
\end{enumerate}
\end{thm-bb}
Another consequence of  the surgery theorem  is the stability under connected sums of the condition $\bRiem(M^n)>t$ for the values of $t$ between $0$ and $(n-1)(n-2)/2$. The same theorem guarantees the stability under surgeries of codimensions $n-1$ or $n$ of $\bRiem(M^n)>t$ for the values of $t$ between $0$ and $(n-2)(n-3)/2$.  As a consequence of this last result, we prove that for $t\in [0, \frac{(n-2)(n-3)}{2})$, the condition $\bRiem(M)>t$ does not impose any restrictions on the fundamental group of $M$.
\begin{cor-b}\label{fund-group} Let $\pi$ be a finitely presented group. Then
 for every $n\geq 4$ and for every $t\in [o, \frac{(n-2)(n-3)}{2})$, there exists a compact $n$-manifold $M$ with
$\bRiem(M)>t$ and $\pi_1(M)=\pi$.\\
In other words, the condition $\bRiem(M^n)\geq \binom{n-2}{2}$ does not impose any restrictions on the fundamental group of a compact $n$-manifold for $n\geq 4$.
\end{cor-b}

It turns out that  the condition $\bRiem(M^n)\geq \binom{n-1}{2}$ (that is the threshold  value  allowed for surgeries!), imposes many restrictions on the topology of the manifold.  The first obstruction result is a vanishing theorem of Betti numbers of the manifold

\begin{thm-c}\label{thmc}
Let $M$ be a compact connected manifold with dimension $n\geq 3$. Then one has
\begin{enumerate}
\item $\bRiem(M^n)>\binom{n-1}{2}\implies b_1(M)=0.$
\item If $\bRiem(M^n)=\binom{n-1}{2}$ and it is attained, then $ b_1(M)\leq n-1.$ 
\end{enumerate}
Here $b_1(M)$ denotes the first Betti number of $M$. 
\end{thm-c}
A direct consequence of the above theorem is the following
\begin{cor-c} Let $n\geq 3$.
\begin{enumerate}
\item[a)]  For any compact manifold $M$ of dimension $n-1$, one has
$$\bRiem(M\times S^1)\leq \binom{n-1}{2}.$$
\item[b)] $\bRiem(S^{n-1}\times S^1)= \binom{n-1}{2}.$ In particular, the standard product metric is maximal.
\item[c)] $\bRiem(\#_r (S^{n-1}\times S^1)= \binom{n-1}{2}.$ There are no maximal metrics on these manifolds if $r\geq n$.

\end{enumerate}
\end{cor-c}

In this paper we prove generalizations of the above theorem \rm{C} to all the higher Betti numbers. In particular, we prove that
\begin{thm-cc}
Let $M$ be a compact connected manifold with dimension $n\geq 4$ and $b_k(M)$ denotes the $k$-th betti number of $M$. Then one has
$$\bRiem(M)>\frac{n(n-1)}{2}-(n-2)\implies b_k(M)=0, \,\,\, {\rm for}\,\, 1\leq k\leq n-1.$$
In particular, $M$ is a homology sphere.
\end{thm-cc}

The next theorem is a weaker version of a theorem due to  B\"ohm and Wilking \cite{BW} and a second theorem due to Ni-Wu \cite{NiWu}. 
It shows in particular that no compact $n$-manifolds, $n\geq 3$, can have their $\bRiem$ within the interval $(N-2, N)$ where $N=n(n-1)/2$.

\begin{thm-d}
Let $M$ be a compact manifold of dimension $n\geq 3$. 
\begin{enumerate}
\item  If $\bRiem(M)>\frac{n(n-1)}{2}-2$ then $\bRiem(M)=\frac{n(n-1)}{2}$.\\
Moreover,  $M$ has a metric $g$ with constant sectional curvature and $\bRiem(M)=\Riem(g)$.
\item
If $\bRiem(M)=\frac{n(n-1)}{2}-2=\Riem(g)$ for some Riemannian metric $g$ on $M$, then either $M$ is locally symmetric or its universal cover is isometric to a product.
\end{enumerate}
\end{thm-d}
The following theorem, which is a consequence of a recent result by Brendle-Hirsch-Johne \cite{BHJ}, shows that taking a cartesian product with a torus does not increase the $\bRiem$. This result combined with the ideal property allowed us to compute the $\bRiem$ invariant for  products of  spheres with tori. Precisely we have

\begin{thm-e}
Let $n\geq 3$ and $1\leq p\leq n-1$ be two integers. Let  $N^{n-p}$ be an arbitrary compact manifold of dimension $n-p$ and $T^p$ be the $p$-dimensional torus. Then for $n\leq 7$ one has
$$\bRiem(N^{n-p}\times T^p)\leq \binom{n-p}{2}.$$
In particular, if $S^{k}$ denotes a positive spherical space form of dimension $k$ then one has
\begin{equation*}
\begin{split}
\bRiem(S^{n-p}\times T^p)&=\bRiem(S^{n-p})= \binom{n-p}{2},\\
\bRiem(S^{2}\times T^{n-2})&=\bRiem(S^{2})= 1.\\
\end{split}
\end{equation*}
\end{thm-e}
As a consequence of recent results by Chen \cite{Chen} and Xu \cite{Xu}, we were able to prove the following rigidity result
\begin{thm-f}
If $n\leq 6$ and $\bRiem(N^{n-p}\times T^p)= \binom{n-p}{2}$ is attained by some Riemannian metric on $N^{n-p}\times T^p$ then 
$N^{n-p}$ has a metric with positive constant curvature.
\end{thm-f}
The following theorem follows from a recent result by Chen \cite{Chen}.
\begin{thm-g}
Let $M^n$ be a compact manifold of dimension $n\leq 7$.
\begin{itemize}
\item[a)] 
If $\bRiem(M^n)\geq 1$ then 
$$\bRiem\Bigl((S^2\times T^{n-2})\# M^n\Bigr)=1.$$
\item[b)]
If $\bRiem(M^n)\geq \binom{n-1}{2}$ then 
$$\bRiem\Bigl((S^{n-1}\times S^1)\# M^n\Bigr)=\binom{n-1}{2}.$$
\end{itemize}
\end{thm-g}

\subsection{Plan of the paper}
Theorem {\bf A} is proved in section 3 and Theorems {\bf B} and ${\mathbf B}^{\prime}$ in section 4. Theorems {\bf C} and {\bf $C^{\prime}$} are proved in section 5 and  the proofs of Theorems {\bf D, E, F} and {\bf G} are in section 2.\\
In the last section 6, we first show as a consequence of Theorem {\bf E} that the product $S^2\times T^p$, for $p\leq 5$, does not admit any Riemannian metric with both $\sigma_1(R)>0$ and $\sigma_2(R)>0$. Where $\sigma_k(R)$ denotes the $k$-th elementary symmetric function in the eigenvalues of the Riemann curvature operator.  Next, we introduce an analogous but weaker smooth invariant, namely $\briem$ invariant. We prove that it remains unchanged under surgeries in most of codimensions. Then we briefly define the $\bRiem$ of a conformal class of metrics, and we determine the $\bRiem$ of the conformal class of the product metric on the product of  two space forms with opposite signs. In this paper, we included as well several open related questions.

\section{Relations with other curvature conditions: Proof of Theorems { \bf D, E, F} and {\bf G}}
Let $M$ be a compact manifold and $g$ a Riemannian metric on $M$.\\
If $\dim M=2$, then $\Riemt(g)=(1-t)\Scal \frac{g^2}{2}$ is determined by the scalar curvature and $\bRiem(M)$ is either $0$ or $1$.
 In $3$ dimensions, it is determined by the Ricci curvature. It is easy to see that in three dimensions the eigenvalues of  $\Riemt(g)$ are equal to  $(1-t)\Scal +2t\rho_i$, for $i=1,2,3$, and  where $\rho_i$ is an eigenvalue of Ricci curvature operator. In particular, 
$\bRiem(M)>1\implies \Ric >0.$ Consequently, by the classification of PSC compact $3$-dimensional manifolds,  $\bRiem(M)$ can take only the values $0,1$ and $3$. \\
In higher dimensions, still we don't know whether $\bRiem$ takes only integer values. As a first step in this direction, we will show in this paper that there are  gaps in the range of $\bRiem$ in all dimensions.\\

The next proposition shows a relation between our curvature condition and positive Ricci curvature condition and $k$-positive Riemann tensor.

\begin{proposition}\label{k-positive}
Let $g$ be a Riemannian metric on a compact $n$-manifold and let $N=n(n-1)/2$. 
\begin{enumerate}
\item For any integer $t\in (0, N)$, the tensor $\Riemt(g)$ is $t$-positive (resp. $t$-nonnegative) if and only if the Riemann curvature operator $R$ of $g$ is $(N-t)$-positive (resp. $(N-t)$-nonnegative). In particular,
\begin{itemize}
\item The Riemann curvature operator $R$ is positive (resp. nonnegative) if and only if  $\Riem_{(N-1)}{(g)}$ is $(N-1)$-positive (resp. $(N-1)$-nonnegative).
\item The Riemann curvature operator $R$ is $k$-positive (resp. $k$-nonnegative) if and only if $ \Riem_{(N-k)}(g)$ is $(N-k)$-positive (resp. $(N-k)$-nonnegative).
\end{itemize}
\item $\Riem(g)>\frac{(n-1)(n-2)}{2}\implies \Ric >0$.
\item $\Riem(g)= \frac{(n-1)(n-2)}{2}\implies \Ric \geq0$.
\end{enumerate}

\end{proposition}
\begin{proof}
To prove the first part, we take the sum of arbitrary $t$ eigenvalues of $\Riemt(g)$ and get
$$t\Scal-2t\sum_{i\in I}\lambda_i=2t\sum_{i=1}^N\lambda_i-2t\sum_{i\in I}\lambda_i=2t\sum_{i\in I^c}\lambda_i,$$
where $N=n(n-1)/2$, $I$ is a set of indices in $\{ 1,2, ..., N\}$ of length $t$ and $\lambda_i$ are the eigenvalues of the Riemann curvature operator. For the second part, let $e_1$ be a given unit tangent vector and complete by $\{e_2,...,e_n\}$ to get an orthonormal basis of the corresponding tangent space, then for $t=\frac{(n-1)(n-2)}{2}$ one has
\[\sum_{i,j\not= 1}\Riem_t(e_i,e_j,e_i,e_j)=(n-1)(n-2)\Scal-2t\bigg(\Scal-2\Ric(e_1,e_1)\bigg)=4t\Ric(e_1,e_1).\]
Finally, if  $\Riem(g)= \frac{(n-1)(n-2)}{2}$ then it is easy to show that $\Riemt(g)\geq 0$ for $t=\frac{(n-1)(n-2)}{2}$ and the last result follows.
\end{proof}
\subsection{Proof of  Theorem {\bf D} }

\begin{proof}
Let $N=\frac{n(n-1)}{2}$. By  assumption there exists a Riemannian metric $g$  on $M$ such that $\Riem(g)>N-2$. Then $\Riem_{N-2}(g)>0$, in particular it is $(N-2)-positive$. Proposition \ref{k-positive} shows then  that the the Riemann curvature operator is $2$-positive. The theorem follows from B\"ohm-Wilking Theorem 1 in \cite{BW}. \\
For the second part, let $\bRiem(M)=\frac{n(n-1)}{2}-2=\Riem(g)$, for some Riemannian metric $g$ on $M$. Consequently one has
$\Riem_{N-2}(g)\geq 0$. Proposition \ref{k-positive} shows  that the Riemann curvature operator is 2-nonnegative. Then by a theorem of Ni-Wu \cite{NiWu}, $M$ is locally symmetric or its universal cover is isometric to a product.
\end{proof}
\subsection{Positive $\mathcal{C}_p$ intermediate curvature condition: Proof of Theorems {\bf E,F} and {\bf G}}
Brendle, Hirch and Johne defined in \cite{BHJ} the notion of  $p$-intermediate curvature denoted $\mathcal{C}_p$. It is defined on the Grassmannian of tangent $p$-planes of a Riemannian $n$-manifold as follows for $1\leq p\leq n-1$. For a tangent $p$-plane $P$, let $\{e_1,...,e_n\}$ be any orthonormal basis of the tangent space such that the vectors $\{e_1,...,e_p\}$ form an orthonormal basis of $P$ and set
$$\mathcal{C}_p(P):=\sum_{i=1}^p\sum_{j=i+1}^n\Sec(e_i,e_j),$$
where $\Sec(e_i,e_j)$ denotes the sectional curvature of the plane spanned by $\{e_i,e_j\}$. Note that $\mathcal{C}_1$ is the Ricci curvature and $\mathcal{C}_{n-1}$ is one half of the scalar curvature.\\
A simple manipulation for $p\leq n-2$ shows that
\begin{equation}
\begin{split}
\mathcal{C}_p(P)=&\sum_{i=1}^n\sum_{j=i+1}^n\Sec(e_i,e_j)-\sum_{i=p+1}^n\sum_{j=i+1}^n\Sec(e_i,e_j)\\
=& \sum_{i,j=1\atop i<j}^n \Sec(e_i,e_j)-\sum_{i,j=p+1\atop i<j}^n \Sec(e_i,e_j)\\
=& \frac{1}{2}\Scal-\frac{1}{2}s_p(P),
\end{split}
\end{equation}
where $s_p(P)$ is the $p$-curvature of the plane $P$ \cite{agag}.\\
The next theorem shows  the existence of a metric with positive $\mathcal{C}_p$ curvature on a given $n$-manifold $M$ provided that  $\bRiem(M)> \frac{(n-p)(n-p-1)}{2}$.

\begin{theorem}\label{riemcp}
\begin{enumerate}
\item For integers $n\geq 3$ and $1\leq p\leq n-2$, one has
$$\bRiem(M) > \binom{n-p}{2}\implies \mathcal{C}_p>0.$$
\item If $\bRiem(M) = \binom{n-p}{2}$ and it is attained then $\mathcal{C}_p\geq 0$.
\end{enumerate}
\end{theorem}
We recall here the equivalence $\bRiem(M) > 0\iff \mathcal{C}_{n-1}>0.$
\begin{proof} We first prove part 1). By assumption, there exists a metric $g$ on $M^n$  with $\Riem(g)>\binom{n-p}{2}$, therefore $\Riem_t(g)>0$ for $t=\binom{n-p}{2}$. Let $\{e_1,e_2..., e_n\}$  be an orthonormal basis of the tangent space at some point and denote by $P$ the plane spanned by the vectors $\{e_1,...,e_p\}$. Then
\begin{equation*}
\begin{split}
\sum_{i,j= p+1}^n&\Riem_t(e_i,e_j,e_i,e_j)=(n-p)(n-p-1)\Scal -2t\sum_{i,j=p+1}^nR(e_i,e_j,e_i,e_j)\\
=&(n-p)(n-p-1)\bigl( \Scal-s_p(P) \bigr) = (n-p)(n-p-1)\mathcal{C}_p(P)>0.\\
\end{split}
\end{equation*}
For the second part, there exists a metric $g$ on $M^n$  with $\Riem(g)= \binom{n-p}{2}$, therefore $\Riem_t(g)\geq0$ for $t=\binom{n-p}{2}$. Using the same argument as above one gets $\mathcal{C}_p\geq 0$. 
\end{proof}
\subsubsection{Proof of Theorem {\bf E}}
\begin{proof}
For $n\leq 7$, the product $N^{n-p}\times T^p$ has no Riemannian metrics with positive $\mathcal{C}_p$  curvature by a recent result of 
 Brendle-Hirch-Johne   in \cite{BHJ}.  The above theorem \ref{riemcp} in this paper implies that $\bRiem(N^{n-p}\times T^p)\leq \frac{(n-p)(n-p-1)}{2}$. To prove the second part recall that the $\Riem$ of standard  metric on $S^{n-p}$ equals $\frac{(n-p)(n-p-1)}{2}$. The ideal property of proposition  \ref{idealp} shows that $\bRiem(S^{n-p}\times T^p)\geq \frac{(n-p)(n-p-1)}{2}$. This completes the proof.
\end{proof}
\begin{remark}
The natural question whether the previous Theorems {\bf E} remains true for all $n\geq 8$,  is an  open question. This is clearly true for $p=1$ and $p=n-1$ by Theorem {\bf C}. 
\end{remark}

\subsubsection{Proof of Theorem {\bf F}}
\begin{proof}
By assumption, there exists a metric $g$ on the product $N^{n-p}\times T^p$  with $\Riem(g)=\binom{n-p}{2}.$ 
 The second part in the above theorem \ref{riemcp}  implies that $\mathcal{C}_p\geq 0$. A recent result by Chu-Kwang-Lee \cite{CKL}for $n\leq 5$, improved  by Xu \cite{Xu} to $n=6$, shows that the metric $g$ must be isometric to a product  of a metric $g_1$ on $N$ and the flat metric on the torus. Therefore one has
$$\Riem(g)=\Riem(g_1)=\binom{n-p}{2}.$$
Consequently $N$ is a positive space form by the above Theorem {\bf D}.
This completes the proof.
\end{proof}

\subsubsection{Proof of Theorem {\bf G}}
\begin{proof}
The surgery theorem shows that the $\bRiem$ of the manifold in part a) is $\geq 1$ while it is $\geq \binom{n-1}{2}$ for the manifold in part b).
The $\bRiem$ can't take higher values in both cases as this will imply the positivity of $\mathcal{C}_{n-2}$ in part a) and the positivity of $\mathcal{C}_1$ in part b). These are not possible by a recent result of Chen \cite{Chen} for $n\leq 7$.
\end{proof}

\subsubsection{Gromov's macroscopic dimension conjecture and the $\bRiem$ constant}
The previous theorem {\bf E} shows clearly an interaction between Gromov's macroscopic dimension of the universal cover of the manifold and its $\bRiem$ invariant. This suggests the following refined conjecture \\
\begin{conjecture} \emph{Let ${\rm dim}_{mc}(M)$ denotes the macroscopic dimension of the universal cover of a compact manifold $M$ and $d$ be an integer such that $1\leq d\leq n-1$. Then}
\[\bRiem(M)> \frac{(n-d)(n-d-1)}{2} \implies {\rm dim}_{mc}(M)< d.\]
\end{conjecture}
Gromov's original macroscopic dimension conjecture is recovered for $d=n-1$. For $d=1$, the above conjecture is trivially true as $\bRiem(M)> \frac{(n-1)(n-2)}{2} \implies \Ric >0$ and therefore ${\rm dim}_{mc}(M)=0$ as its universal cover is in this case compact. The conjecture is as well true for $d=2$ and $n\leq 5$. In fact, the condition $\bRiem(M)> \frac{(n-2)(n-3)}{2} \implies \mathcal{C}_2 >0$. Where the condition $\mathcal{C}_2 >0$ is as in the above section. The later condition is shown by Xu \cite{Xu} to imply that ${\rm dim}_{mc}(M)<2$.

\section{The $\bRiem$ of total spaces of Riemannian submersions: Proof of Theorem {\bf A} and Proposition \ref{idealp}}

\subsection{Proof of Proposition \ref{idealp}}
\begin{proof}
We will show that the $\bRiem$ of the Cartesian product of two manifolds  $M_1$ and $M_2$ cannot be less than the $\Riem$ of $M_1$ or $M_2$.\\ 
If $\bRiem(M_1)=\bRiem(M_2)=0$, the result is trivial. If  $\bRiem(M_1)>0$, then $M_1$ possesses a metric $g_1$ with $\Riem(g_1)>0$. Then one can amplify the metric $g_1$ by multiplying it by $t>0$ and use the compactness to show that  $\bRiem(M_1\times M_2) \geq  \bRiem(M_1).$ The proof can be then completed easily after considering the cases $\bRiem(M_2)=0$ and $\bRiem(M_2)>0$.
\end{proof}

\subsection{Proof of Theorem {\bf A}}

Theorem {\bf A} is a special case of the following more general theorem
\begin{theorem}
Let $M$ be compact and be the total space  of a Riemannian submersion $\pi:(M,g)\to (B,\check{g})$ with totally geodesic fibers $F_x, x\in B$. Let $p=\dim F_x$ and $k\in [0, p(p-1)/2)$.  Suppose that the induced metric on the fibers satisfy $\Riem(g_{|_{F_x}})>k$ for all $x\in B$ then $\bRiem(M)>k$.
\end{theorem}

\begin{proof}
We use the canonical variation $g_t$ of the metric $g$, that is the metric obtained by  re-scaling $g$ in the vertical directions by $t^2$. We shall prove that there exists $t>0$ such that $\Riem(g_t)>k$, and consequently one has $\bRiem(M)>k$ as desired.\\
Let $U,V,W,W^{\prime}$ be vertical vectors of $g_t$-length 1 and $X,Y,Z,Z^\prime$ be arbitrary forizontal vectors. We shall index by $t$ all the invariants of the metric $g_t$ and put under a hat the invariants of the fibers with the induced metric. We omit the index $t$ in case $t=1$. O'Neill's formulas for Riemannian submersions show that, see for instance chapter 9 in \cite{Besse} or chapter 2 in \cite{Labbi-these}
%\begin{equation*}
\[
\begin{split}
R_t&(U\wedge V,W\wedge W^\prime)=t^2\hat{R}(U\wedge V,W\wedge W^\prime),\\
R_t&(U\wedge V,W\wedge X)=0,\\
R_t&(X\wedge U,Y\wedge V)= O(1),\\
R_t&(U\wedge V,X\wedge Y)= O(1),\\
R_t&(X\wedge Y,Z\wedge U)= o(t),\\
R_t&(X\wedge Y,Z\wedge Z^\prime)= O(1).
\end{split}
%\end{equation*}
\]

Here the Riemann tensor is seen as a $(2,2)$ double form, $\hat{R}$ is the double form associated to the Riemann tensor of the fibre metric $\hat{g}=g_{|_{F_x}}$.\\
Let now $\phi$ be any 2-form of $g_t$-unit length in $\wedge^2M$ and denote by $\hat{\phi}$  its pointwise orthogonal projection onto $\wedge^2F$, then the above formulas show that
\[R_t(\phi,\phi)=t^2\hat{R}(\hat{\phi},\hat{\phi})+O(1).\]
Suppose the maximum eigenvalue of the curvature operator $R_t$ is $\lambda^t_{\rm max}=R_t(\phi,\phi)$ for some $g_t$-unit length 2-form $\phi$. Then one has at each point of $M$ the following
\[\lambda^t_{\rm max}=\frac{1}{t^2}\hat{R}(t^2\hat{\phi},t^2\hat{\phi})+O(1)\leq \frac{1}{t^2}\hat{\lambda}_{\rm max}+O(1).\]
Here $\hat{\lambda}_{\rm max}$ denotes the maximum eigenvalue of the curvature operator $\hat{R}$. In the last argument, we used the fact that $||t^2\hat{\phi}||=||\hat{\phi}||_t\leq  ||\phi||_t=1$.\\
To complete the proof just remark that
\[\frac{\Scal(g_t)}{2\lambda^t_{\rm max}}\geq \frac{\Scal(\hat{g})+O(t^2)}{2\hat{\lambda}_{\rm max}+O(t^2)}.\]
Where we used the fact that $\Scal(g_t)=\frac{\Scal(\hat{g})}{t^2}+O(1)$.  Consequently, at each point of $M$,  if $k<\frac{\Scal(\hat{g})}{2\hat{\lambda}_{\rm max}}$ then there exists $t>0$ such $k<\frac{\Scal(g_t)}{2\lambda^t_{\rm max}}$. We conclude using the compactness of $M$.
\end{proof}
%\begin{corollary}\label{n2t2}
%If $N$ is an arbitrary compact manifold of dimension $2$ then
%$$\bRiem(N\times T^2)\leq 1.$$
%\end{corollary}

\section{The $\bRiem$ invariant and surgeries: Proof of Theorems {\bf B}, ${\mathbf B}^{\prime}$ and Corollary {\bf B}}

\subsection{Proof of Theorem {\bf B}}
The following theorem is a reformulation of Theorem {\bf B}.
\begin{theorem}
Let  an integer $q$ be such that $3\leq q\leq n$ and the real number $t$ satisfies $0\leq t<\binom{q-1}{2}.$\\
The condition $\Riem(g)>t$ on  Riemannian compact $n$-manifolds is preserved under  surgeries of codimension $\geq q$.
 
\end{theorem}

\begin{proof}
Recall that the condition $\Riem(g)>t$ is equivalent to requiring the positivity of the tensor $\Riemt(g)=\Scal(g)\frac{g^2}{2}-2tR$. The theorem follows directly from Hoelzel's general surgery theorem  \cite{Hoelzel}. We  use the same notations as in \cite{Hoelzel}. Let $C_B({\Bbb{R}}^n)$ denote the  vector space of algebraic curvature operators $\Lambda^2 {\Bbb{R}}^n\rightarrow \Lambda^2 {\Bbb{R}}^n$ satisfying the first Bianchi identity and endowed with the canonical inner product.
For $0<t<n(n-1)/2$, let 
$$C_{\Riemt >0}:=\{R\in C_B({\Bbb{R}}^n): \Riemt(R)>0\},$$
here $\Riemt(R)=\Scal(R)\frac{g^2}{2}-2tR$.  The subset $C_{\Riemt >0}$ is clearly open, convex and it is an ${\rm{O}}(n)$-invariant cone. Furthermore, it is easy to check that $\Riemt(S^{q-1}\times \Bbb{R}^{n-q+1})>0$ if   $t<\frac{(q-1)(q-2)}{2}$ and $q\geq 3$. This completes the proof.

\end{proof}

We come now to the proof of Theorem {\bf B} as follows.
\begin{proof}
First, we prove part a) as follows. Let $0\leq t<\bRiem(M)$. Then there exists a Riemannian metric $g_t$ on $M$ such that $\Riem(g_t)>t$. Since $t<\bRiem(M)\leq \binom{q-1}{2}$ and the codimension of the surgery is $\geq q$, the above theorem proves the existence of a metric $\hat{g}_t$ on $\widehat{M}$ satisfying $\Riem(\hat{g}_t)>t$. The part a) of the  theorem follows directly as $t$ was chosen arbitrary in the above interval.\\
Next we prove the second part in a similar way.
Let $ t<\binom{k}{2}$. Then $ t<\bRiem(M)$ and there exists a Riemannian metric $g_t$ on $M$ such that $\Riem(g_t)>t$. Since $t<\binom{(k+1)-1}{2}$ and the codimension of the surgery is $\geq k+1$, the above theorem proves the existence of a metric $\hat{g}_t$ on $\widehat{M}$ satisfying $\Riem(\hat{g}_t)>t$. This completes the proof of the theorem.
\end{proof}

The following corollary follows directly from the above Theorem
\begin{corollary}\label{surgery-n-1}
\begin{enumerate}
\item[a)] Let $M_1$ and $M_2$ be two compact manifolds of dimensions $n\geq 3$ and such that $0<\bRiem(M_1)\leq \frac{(n-1)(n-2)}{2}$ and $\bRiem(M_2)\geq \bRiem(M_1)$. Then their connected sum  satisfies
\[\bRiem(M_1\#M_2)\geq \bRiem(M_1).\]
\item[b)] Let $M$ be a compact $n$-manifold with $n\geq 4$ and $0< \bRiem(M)\leq \frac{(n-3)(n-2)}{2}$. If a compact manifold $\widehat{M}$ is obtained from $M$ by surgeries of 
codimensions $\geq n-1$ then $\bRiem(\widehat{M})\geq \bRiem(M)$. \\

\end{enumerate}
\end{corollary}
\subsection{Proof of Corollary {\bf B}}
\begin{proof}
Corollary {\bf B} results from the above corollary \ref{surgery-n-1}. The proof follows word by word the proof of an analogous result for the $\Ein$ invariant in \cite{modified-Eink}. We will not reproduce it here.
\end{proof}
\subsection{Proof of Theorem ${\mathbf B}^\prime$}
\begin{proof}
We prove the first part as follows. 
Let $M$ be a compact simply connected spin (resp. non-spin) manifold of dimension $\geq 5$. Gromov and Lawson \cite{GroLa}  proved that if $M$ is spin cobordant (resp. oriented cobordant) to a compact manifold $M_1$ then $M$ can be obtained from $M_1$ by surgeries of codimensions $\geq 3$. \\
From another side, according to a result by F\"uhr \cite{Fuhr},  a closed oriented manifold of dimension $\geq 5$ is always oriented cobordant to the total space $M_1$ of $\mathbb{CP}^2$ bundle with structure group the isometry group of the Fubiny-Study metric of $\mathbb{CP}^2$. Theorem ${\bf A}$ shows that $\bRiem(M_1)\geq 2$. It follows from the surgery theorem {\bf B} part b) that $\bRiem(M)\geq 1$. This completes the proof for the non-spin case.\\
If $M$ is spin  with $\bRiem>0$ then it has a positive scalar curvature metric. A theorem of 
 Stolz \cite{Stolz} guarantees that $M$ is spin cobordant  to the total space $M_1$ of an $\mathbb {HP}^2$  bundle with structure group the isometry group of the Fubiny-Study metric of $\mathbb {HP}^2$. Theorem ${\bf A}$ shows that $\bRiem(M_1)\geq 8$ and then $\bRiem(M)\geq 1$ by  the second part of the surgery theorem {\bf B}. This completes the proof of the first part.\\
Next, we prove the second part of the Theorem. Let $M$ be compact and $2$-connected  with $\bRiem>0$ and dimension $\geq 7$. Then $M$ has a canonical spin structure and has a metric with positive scalar curvature. The above mentioned theorem of Stolz shows that $M$ is spin cobordant  to a manifold $M_1$ 
with $\bRiem(M_1)\geq 8$. On the other hand,  a previous result of the author \cite{agag} asserts that $M$  can be obtained from $M_1$ by surgeries of codimension $\geq 4$. The part b) of the surgery theorem {\bf B} shows that $\bRiem(M)\geq 3$.\\
Finally, we prove the last part of the theorem. 
The manifold $M$ is $3$-connected then it has a canonical spin structure and it has a positive scalar curvature metric since $\bRiem(M)>0$.  As above, a consequence of Stolz's result is that  $M$ is spin cobordant  to a compact manifold $M_1$ with $\bRiem(M_1)\geq 8$. On the other hand,  a previous result of Botvinnik and  the author \cite{Bot-Lab1} asserts that if $M$ is  compact $3$-connected  and non-string of dimension $n\geq 9$ that is spin cobordant to a manifold $M_1$, then the manifold $M$ can be obtained from $M_1$ by surgeries of codimension $\geq 5$. The second part of the surgery theorem implies that $\bRiem(M)\geq \binom{4}{2}=6.$ This completes the proof of the Theorem.
\end{proof}

\section{Vanishing theorems: Proof of Theorems  $\mathbf{C}$ and ${\mathbf C}^\prime$}
\subsection{Proof of Theorem {\bf C} and Corollary {\bf C}}
\begin{proof}
Let $s=\frac{(n-1)(n-2)}{2}$. The condition $\bRiem(M)>s$ guarantees the existence of a Riemannian metric $g$ on $M$ such that $\Riem(g)>s$. It follows from Proposition \ref{k-positive} that the Ricci curvature of $g$ is positive and therefore the Betti numbers $b_1$ of $M$ vanish. Next, if $\bRiem(M)=\Riem(g)=s$ for some Riemannian metric $g$ on $M$, then Proposition \ref{k-positive} shows that the Ricci curvature of $g$ is nonnegative and therefore $b_(M)\leq n$ by Bochner theorem. Since here the $\bRiem$ is supposed to be positive then the scalar curvature is positive, so our manifold can't be  a torus and therefore $b_1(M)\leq n-1$ again by Bochner theorem. This completes the proof of Theorem {\bf C}.\\
Next we prove the corollary. Part a) results directly from the theorem as  $b_1(M\times S^1)\not=0$. To prove part b), we 
recall that the standard product metric on $S^{n-1}\times S^1$ has $\Riem$ equal to $s$.  We prove part {\bf c)}  in a similar way. From one side  corollary \ref{surgery-n-1} shows that $\bRiem\big(\#_r(S^{n-1}\times S^1)\big)\geq s$. From another side it cannot be strictly higher than $s$ because of the above positive Ricci curvature obstruction.
\end{proof}

\subsection{Proof of Theorem  ${\mathbf C}^\prime$}
We prove the following more general version of Theorem  ${\mathbf C}^\prime$.
\begin{theorem}\label{gen-bk}
Let $M$ be a compact connected manfold with dimension $n\geq 3$ and let $p$ be an integer such that 
$2\leq p\leq n-2$. We denote by $b_k(M)$ as usual the $k$-th betti number of $M$, then one has
$$\bRiem(M)>\frac{n(n-1)}{2}-\frac{p(n-p)}{2}\implies b_k(M)=0, \,\,\, {\rm for}\,\, p\leq k\leq n-p.$$
\end{theorem}
\begin{proof}
We shall use double forms formalism for the Weitzenb\"ock cuvature term \cite{Labbi-nachr}. The curvature term in Weitzenb\"ock formula once applied to $p$-forms takes the form
\[ W_p=\frac{g^{p-1}}{(p-1)!}\Ric-2\frac{g^{p-2}}{(p-2)!}R.\]
On the other hand we have
\[\frac{g^{p-2}}{(p-2)!}\Riemt=\frac{g^{p-2}}{(p-2)!}\bigg( \Scal \frac{g^2}{2}-2tR\bigg)=
\frac{p(p-1)\Scal}{2}\frac{g^p}{p!}-2t\frac{g^{p-2}}{(p-2)!}R.\]
Consequently, for $t>0$ we get 
\[tW_p=\frac{g^{p-2}}{(p-2)!}\Riemt+\frac{g^{p-1}}{(p-1)!}\bigg(t\Ric-\frac{(p-1)\Scal}{2}g\bigg).\]

We will show that under the theorem hypothesis where $t>\frac{n(n-1)}{2}-\frac{p(n-p)}{2}$ the second term in the last sum is posiitve as well. This amounts to proving that the sum of the lowest $p$ eigenvalues of $t\Ric-\frac{(p-1)\Scal}{2}g$ is positive.  Recall that the condition $\bRiem(M)>t$ implies the existence of a Riemannian metric on $M$ with $\Riemt>0$. After taking the trace, one can see that $(n-1)\Scal\,  g-2t\Ric>0$. Taking the sum of $(n-p)$ eigenvalues of the later, we see that 
\[(n-1)(n-p)\Scal-2t(\Scal -\sum_{\i\in I}\rho_i)>0,\]
where $\rho_i$ denotes the eigenvalues of Ricci and $I\subset \{1,2,...,n\}$ is any subset of indices of length $p$. Therefore, we get
\[2t\sum_{\i\in I}\rho_i-p(p-1)>2t-(n-1)(n-p)-p(p-1)>0.   \]
This completes the proof of the theorem.

\end{proof}

\begin{remark}
Let $t=\frac{(n-1)(n-2)}{2}+p$.
The condition $\bRiem(M)>t$ implies the existence of a Riemannian metric on $M$ with $\Riemt>0$. Proposition \ref{k-positive} implies then that the Riemann curvature tensor of $g$ is $(\frac{n(n-1)}{2}-t)$-positive, that is $(n-(p+1))$-positive. Since $ p\leq \frac{n-2}{2}$ then $p+1\leq n/2$ and therefore a vanishing theorem of  Petersen-Wink \cite{Pet-Win} shows the vanishing of Betti numbers $b_k(M)=0$ for $1\leq k\leq p+1$ and 
$n-p-1\leq k\leq n-1.$ 
\end{remark}

\section{Miscellaneous results}

\subsection{Positive $\Gamma_2(R)$ curvature}

Let $\sigma_1(R)$ and $\sigma_2(R)$ be the first two elementary symmetric functions in the eigenvalues of the Riemann curvature operator. If $\sigma_1(R)>0$ and $\sigma_2(R)>0$, we shall write $\Gamma_2(R)>0$.
Note that $\sigma_1(R)=\frac{\Scal(R)}{2}$ and $8\sigma_2(R)=\Scal^2(R)-4||R||^2$. In particular, 
$$\Gamma_2(R)>0\iff \Scal(R)>2||R||.$$
 The following theorem shows in particular that the product $S^2\times T^p$ has positive scalar curvature but does not allow at the same time both $\Scal>0$ and $\sigma_2(R)>0$.
\begin{theorem}
If a Riemannian manifold $(M,g)$ has positive $\Gamma_2(R)$ curvature then $\bRiem(M)>1$. In particular, the product $S^2\times T^p$  does not support any metric with  positive $\Gamma_2(R)$ curvature (at least for $p\leq 5$). The same is true for the connected sum $(S^2\times T^{p})\# M$ for any compact manifold $M$ of dimension $p+2\leq 7$.
\end{theorem}

\begin{proof}
We remark that the first Newton transformation of the curvature operator is $t_1(R)=\frac{\Scal}{4}g^2-R=\frac{1}{2}(\Riem_1(g)$. It is classic that if $\sigma_1$ and $\sigma_2$ of an operator are positive then its first Newton transformation is positive. Then $\Riem_1(g)>0$, that is $\Riem(g)>1$. This cannot take place for the products $S^2\times T^p$ with $p\leq 5$ by Theorem {\bf E} and for the connected sum as in the theorem by Theorem {\bf G}.  This completes the proof.
\end{proof}
\begin{remark}
The author believes that Dirac operator techniques may help to prove that the positivity of $\Scal -2||R||$ is not allowed on products of $S^2\times T^q$, for all $q$, these are somehow partially enlargeable manifolds.
\end{remark}

\subsection{The small $\Riem$ invariant}
For a fixed Riemannian metric $g$ on a compact $n$-manifold $M$ and for $s<t<0$, the tensors $\Riemt(g)$ enjoy the following descent property
$$\Riems>0\implies\Riemt>0\implies \Scal >0.$$
We therefore define the metric invariant
\[ \riem(g):=\inf \{t<0:\, \Riemt(g)>0\}.\]
We set it equal to $-\infty$ if the above set is unbounded below and equal to zero if that set is empty.\\
It is not difficult to see that  $\riem(g)=-\infty$ if and only if the Riemann tensor $R$ is nonnegative and with positive scalar curvature. \\
We define the smooth invariant $\briem(M):=\inf \{\riem(g)\colon  g\in {\mathcal M}\},$
where ${\mathcal M}$ denotes the space of all  Riemannian metrics on $M$.
It  is remarkable that this invariant remains unchanged after surgeries, precisely we have
\begin{theorem}
Let $M$ be a compact manifold of dimension $n\geq 4$.
If a compact manifold $\widehat{M}$ is obtained from $M$ by surgeries of 
codimensions $\geq 3$ but not equal to $n-1$ then $\briem(\widehat{M})= \briem(M)$. \\
\end{theorem}

\begin{proof}
Using Hoelzel's surgery theorem \cite{Hoelzel}, it is easy, as in the proof of Theorem B,  to see that  $\briem(\widehat{M}) \geq  \briem(M)$ as far as the codimension of the surgery is $\geq 3$.
To prove equality we apply a reversed surgery as in \cite{petean}. In fact one can recover the initial  manifold $M$ from the new manifold $\widehat{M}$ by applying a surgery of codimension $n-q+1$.  Consequently one gets $\briem(M)\geq \briem(\widehat{M})$ if the new codimension  $n-q+1\geq 3$. Consequently, the $\briem$ is unchanged if $3\leq q\leq n-2$. The case of $q=n$ can be ruled out, as in this cas $\widehat{M}$ is diffeomorphic to the connected sum of $M$ with the product $S^1\times S^{n-1}$.  One can then recover $M$ by killing the circle by the mean of a  surgery of codimension only $n-1\geq 3$. This completes the proof.
\end{proof}
One consequence of this theorem is that simply connected compact PSC manifolds of dimensions $\geq 5$ have their $\briem$ equal to $-\infty$.

\subsection{The $\Riem$ invariant of a conformal class}

Let $[g]$ denotes the  conformal class of the metric $g$. We define the $\Riem$ of the conformal class $[g]$ as
\[\Riem([g])= \sup\{\Riem(g):  g\in [g]\}.\]

As above, we set it equal to zero if the conformal class does not contain a psc metric. We  prove the following vanishing theorem and a consequence of it which determines the conformal $\Riem$ for some conformally flat classes.

\begin{theorem}
Let $(M,g)$ be a compact oriented conformally flat $n$-manifold  and $p$ an integer such that $0<p\leq n/2$.  Then one has
$$\Riem([g])>\frac{(n-1)(n-2p)}{2}\implies b_k(M)=0,\,\,\, {\rm for}\,\, p\leq k\leq n-p.$$
In particular, for any $n>2p\geq 2$, let $g_0$ denotes the product metric on the product of two space forms of opposite signs  $S^{n-p}\times H^p$, then one has
$$\Riem([g_0])=\Riem(g_0)=\frac{(n-1)(n-2p)}{2}.$$
\end{theorem}

\begin{proof}

In what follows,  products of tensors are Kulkarni-Nomizu products.
Recall tor a conformally flat metric $g$ one has $R=gA$, where $A$ is the Schouten tensor. Consequently, the $\Riemt$ tensor is determined by the Ricci tensor as follows
\begin{equation}
\begin{split}
(n-2)\Riemt(g)=&\bigg(\frac{(n-1)(n-2)+2t}{n-1}\bigg)\Scal \frac{g^2}{2}-2tg\,\Ric\\
=& \bigg(\frac{(n-1)(n-2)+2t}{2(n-1)}\bigg)g\,\Ein_T.
\end{split}
\end{equation}
Here, $\Ein_T=\Scal g-T\,\Ric$ and $T=\frac{4t(n-1)}{(n-1)(n-2)+2t}$. At this stage we use a vanishing theorem for the $\Ein$ invariant \cite{Labbi-conformal}, which gurantees the vanishing of the $b_k(M)$ as in the Theorem that we are proving under the condition that $T>\frac{(n-1)(n-2p)}{n-p-1}$. It is straightforward to see that this last condition is equivalent to $t>\frac{(n-1)(n-2p)}{2}$. This proves the first part. For the second part, note that  from one hand the $p$-th Betti number of the above product is not zero therefore $\Riem([g_0])\leq \frac{(n-1)(n-2p)}{2}$. On the other hand the product metric satisfies $\Riem(g_0)=\frac{(n-1)(n-2p)}{2}$. This completes the proof.
\end{proof}
\begin{remark}
It is remarkable here that the conformal $\Riem$ of the above product metric is not a binomial coefficient number of the form $\binom{k}{2}$, The author doesn't know whether  $\bRiem(S^{n-p}\times H^p)$ equals $\frac{(n-1)(n-2p)}{2}.$ 
We ignore as well  whether there exist any compact manifold whose $\bRiem$ is not a binomial coefficient number of the form $\binom{k}{2}$ for some integer $k$.
\end{remark}
\subsection{Minimal vs. Maximal PSC compact manifolds}
The smooth $\bRiem$  invariant defines a {\emph pre-order} on the set of all compact PSC manifolds with a fixed dimenion $n$. The maximal manifolds are by B\"ohm-Wilking theorem (see Theorem ${\bf D}$)  the manifolds with constant positive sectional curvature (space forms). In the other extreme, one may ask the following questions:
 {\emph{What are the PSC compact $n$-manifolds with minimal $\bRiem$ ?}}.  {\emph{ What are the PSC compact simply connected (resp.  2-connected) manifolds with minimal $\bRiem$ ?}}.\\
For instance, these are more specific questions:
\begin{itemize}
\item Are there  compact manifolds with $0<\bRiem<1$?  Is  $S^2\times T^{n-2}$  minimal among  all compact PSC $n$-manifolds? 
\item Are there  compact simply connected manifolds with $0<\bRiem<2$? 
\end{itemize}
The problem is very well understood in dimension 3. In fact, it results from the classification of compact PSC 3-manifolds that the minimal PSC manifolds are those with $\bRiem=1$ and they are either $S^2\times S^1$ or connected sums of copies of the later with spherical space forms.
\smallskip\noindent

\subsection{Range of the $\Riem$ functional}
Let $M^n$ be a compact PSC manifold and denote by $\mathcal{R}^+(M^n)$ the space of PSC metrics on $M^n$ endowed with the usual  $C^\infty$-topology. We consider the functional
\begin{equation*}
\begin{split}
\Riem: \mathcal{R}^+(M^n)&\to \mathbb{R}\\
g&\to \Riem(g).
\end{split}
\end{equation*}
This functional is continuous and defines a topological stratification of the space $\mathcal{R}^+(M^n)$. The range of the above functional  is then an interval in $\mathbb{R}$ once restricted to  connected components in $\mathcal{R}^+(M^n)$. In two dimensions the $\Riem$ functional is constant equal to $1$. One wonder to see what should be the range of this functional for a general compact manifold $M^n$ for $n\geq 3$. The next proposition reveals a small piece of the story
\begin{proposition}
Let $n\geq 2$ and $\mathcal{R}^+(S^n)$ denotes the space of PSC metrics on the standard smooth sphere $S^n$, then one has
$$\Riem\bigl(\mathcal{R}^+(S^{2n-1})\bigr)=(0, N],\,\,\, {\rm with}\,\, N=\binom{2n-1}{2}.$$
\end{proposition}
\begin{proof} We shall use the Berger metrics on odd dimensional spheres.
Let $S^{2n-1}(r)$ denotes the geodesic sphere of radius $r$ in $\mathbb{CP}^n$ with its Fubiny-Study metric $g_{fs}$ and $0<r<\pi/2$. Denote by $g_r$ the induced metric on the embedded spheres, then one can see without difficulties using the computations in \cite{Bour-Kar} that
$$\Riem(g_r)=\frac{4n(n-1)+(2n-2)(2n-1)\cot^2r}{4n+2\cot^2r}.$$
As $r\to \pi/2$,  $\Riem(g_r)$ tends to $n-1$ that is the $\Riem$ of $\mathbb{CP}^{n-1}$ with the Fubini-Study metric. That's not a surprise as the limit submanifold is known to be a totally geodesic submanifold isometric to $\mathbb{CP}^{n-1}$ (conjugate locus). On the other side, as $r\to 0$, $\Riem(g_r)$ tends to $N$ which is the $\Riem$ of the standard round metric on the sphere $S^{2n-1}$. Again this is expected as $\frac{1}{r^2}g_r$ is known to converge $C^2$ to the standard round metric of curvature 1 on $S^{2n-1}$.\\
We remark here that for all $r\in (0, \pi/2)$ the Riemann tensor $R$ is positive but the $\Riem$ values go down till the value of $n-1$ on $S^{2n-1}$.\\
Next, we consider the geodesic spheres $S^{2n-1}(r)$  of radius $r$ in the complex hyperbolic space $\mathbb{CH}^n$ with its Fubiny-Study metric $g_{fs}$. Denote by $g_r$ the induced metric on the embedded spheres, then one can see without difficulties using the computations in \cite{Bour-Kar} that
$$\Riem(g_r)=-2n(n-1)\tanh^2r+(n-1)(2n-1).$$
As $r\to 0$, $\Riem(g_r)$ tends to $N$ which is the $\Riem$ of the standard round metric on the sphere $S^{2n-1}$, this is expected as explained above.  However, as $\tanh^2r$ tends to $\frac{2n-1}{2n}$ from the left, one has $\Riem(g_r)$ tends to $0$ from above.
This completes the proof.
\end{proof}
It is an open question to decide whether in general (or otherwise, for which manifolds) the range of the $\Riem$ functional on a compact PSC manifold $M$ of dimension $\geq 3$ is always the interval from $0$ excluded to $\bRiem(M)$ included or excluded.

\section*{Acknowledgement}
I would like to thank Professor Boris Botvinnik for his interest in this work and for his suggestion to use the term stratification to descibe this work.
\section*{Statements and Declarations}
The author declares that no funds, grants, or other support were received during the preparation of this manuscript.
\section*{Competing Interests}
The author have no relevant financial or non-financial interests to disclose.
\section*{Data Availability}
Data sharing not applicable to this article as no datasets were generated or analysed during the current study.

\end{document}